\documentclass[12pt,a4paper]{article}
\usepackage{mathrsfs}
\usepackage{epsfig, graphicx}
\usepackage{latexsym,amsfonts,amsbsy,amssymb}
\usepackage{amsmath,amsthm}
\usepackage{color}
\usepackage[colorlinks, citecolor=blue]{hyperref}

\makeatletter 

\@addtoreset{equation}{section}
\makeatother 
\allowdisplaybreaks 

\newtheorem{theorem}{Theorem}[section]
\newtheorem{lemma}{Lemma}[section]

\newtheorem{proposition}{Proposition}[section]

\begin{document}
\title{Local a Priori Estimate on the General Scale Subdomains\footnote{This work is supported in part
by the National Science Foundation of China (NSFC 11001259, 11371026, 11201501, 11031006,
2011CB309703), the National Center for Mathematics and Interdisciplinary Science,
 CAS and the President Foundation of AMSS-CAS.}}
\author{Xiaole Han\footnote{LSEC, ICMSEC, Academy of Mathematics and Systems
Science, Chinese Academy of Sciences,  Beijing 100190, China
(hanxiaole@lsec.cc.ac.cn)} \ and\ Hehu Xie\footnote{LSEC, ICMSEC,
Academy of Mathematics and Systems Science, Chinese Academy of
Sciences,  Beijing 100190, China (hhxie@lsec.cc.ac.cn)} }
\date{}
\maketitle
\begin{abstract}
The local a priori estimate for the finite element approximation is essential
for underlying the local and parallel technique. It is well known that the constant
coefficients in the inequality is independent of the mesh size. But it is not so
clear whether the constant depends on the scale of the local subdomains. The aim of
this note is to derive a new local a priori estimate on the general scale domains. We also 
 show that the dependence of the constant appearing in the local
a priori estimate on the scale of the subdomains.


\end{abstract}

\section{Introduction}
Recently, parallel techniques for the finite element computation have become very attractive.
These exists a type of parallel schemes which are based on the understanding of the local
and global properties of a finite element solution
for the elliptic type problems which is proposed in \cite{XuZhLP2G} and then has been 
studied extensively \cite{LiHanXieYouLPMEigen,LinXieMultilevelEigen,XuZhLP2GNL,XuZhLP2GEigen}.
The cornerstone of this technique is the local a priori estimate \cite{Wahlbin,XuZhLP2G},
where the involved constant is independent of the mesh parameters. However, the dependence of
 the constant on the scale of the subdomains is not so clear. If we want to consider the sharp 
 effects caused by local subdomain ($\Omega_0$), it is necessary to clarify the dependence of 
 the coefficient on the diameter of $\Omega_0$ (denoted by $d_{\Omega_0}$).
For example, in \cite{LiHanXieYouLPMEigen,XuZhLP2GNL,XuZhLP2GEigen}, we need to know
how large of the subdomain to construct the efficient parallel method. 
In this paper, we explicitly show the dependence of the local priori error estimate on the 
subdomain scale $d_{\Omega_0}$.

The rest of this paper is organized as follow. In the next section, some notation,
assumptions and basic results are listed.
In Section 3, a local estimate of the finite element solution 
is derived on the general scale domain and the dependence of the local 
estimates on the subdomain scale is clarified. 
Some concluding remarks are given in the last section.

\section{Preliminaries}
In this section, following \cite{XuZhLP2G}, we firstly state the model problem and list
some basic notations and results. Then we set some reasonable assumptions on the
finite element spaces and show their reasonability.

\subsection{Model problem}
Let $\Omega$ be a bounded domain in $\mathcal{R}^d$ $(d\ge1)$. We shall use standard
notation for Sobolev spaces $W^{s,p}(\Omega)$ and their
associated norms and seminorms \cite{Adams}. For $p=2$, we denote
$H^s(\Omega)=W^{s,2}(\Omega)$ and $H_0^1(\Omega)=\{v\in H^1(\Omega):\ v|_{\partial\Omega}=0\}$,
where $v|_{\partial\Omega}=0$ is in the sense of trace, and $\|\cdot\|_{s,\Omega}=\|\cdot\|_{s,2,\Omega}$.
In some places, $\|\cdot\|_{s,2,\Omega}$ should be viewed as piecewise defined if it is necessary.
For $D\subset G\subset\Omega$, the notation $D\subset\subset G$ to means that
${\rm dist}(\partial D\setminus\partial\Omega,\partial G\setminus\partial\Omega)>0$.
Note that any $w\in H_0^1(\Omega_0)$ can be naturally extended to be a function
in $H_0^1(\Omega)$ with zero
outside of $\Omega_0$. Thus we will state this fact by the abused notation
$H_0^1(\Omega_0)\subset H_0^1(\Omega)$.

In this paper, we mainly consider the following second order elliptic problem:
\begin{equation}\label{Model_Problem}
\left\{
\begin{array}{rcl}
\mathcal{L}u & = & f, \ \ \rm{in}\ \Omega, \\
u & = & 0, \ \ \rm{on}\ \partial\Omega.
\end{array}
\right.
\end{equation}
Here $\mathcal{L}$ is a general linear second order elliptic operator:
\begin{equation*}
\mathcal{L}u = -\sum_{i,j=1}^d \frac{\partial}{\partial x_j}\Big(a_{ij}\frac{\partial u}{\partial x_i}\Big)
+ \sum_{i=1}^d b_i \frac{\partial u}{\partial x_i} + \phi u
\end{equation*}
with $a_{ij}$, $b_i \in W^{1,\infty}(\Omega)$, $0\leq \phi \in L^{\infty}(\Omega)$
 and the matrix $(a_{ij})_{1\leq i,j\leq d}$ being
uniformly positive definite on $\Omega$.

The weak form of (\ref{Model_Problem}) is as follows:

Find $u\equiv \mathcal{L}^{-1}f \in H_0^1(\Omega)$ such that
\begin{equation}\label{Weak_Form_Global}
a(u,v) = (f,v), \quad \forall v \in H_0^1(\Omega),
\end{equation}
where $(\cdot,\cdot)$ is the standard inner-product of $L^2(\Omega)$ and
\begin{equation*}
a(u,v) = a_0(u,v) + N(u,v)
\end{equation*}
with
\begin{equation}\label{Bilinear_Form}
a_0(u,v) = \int_{\Omega}\sum_{i,j=1}^d a_{ij}\frac{\partial u}{\partial x_i}\frac{\partial v}{\partial x_j}d\Omega\ \
{\rm{and}} \ \  N(u,v) = \int_{\Omega}\Big(\sum_{i=1}^d b_i \frac{\partial u}{\partial x_i}v + \phi uv\Big)d\Omega.
\end{equation}
In this paper, we assume there exists constants $C$ independent of $\Omega$ such that the follow inequalities hold
\begin{equation}\label{positive_definite_A}
\|w\|_{1,\Omega} \le C a_0(w,w), \ \ \ \forall w \in H_0^1(\Omega),
\end{equation}
and
\begin{equation}\label{Bounded_Operator}
a_0(u,v) \le C \|u\|_{1,\Omega} \|v\|_{1,\Omega},
\ \ N(u,v) \le C \|u\|_{0,\Omega}\|v\|_{1,\Omega}, \ \ \ \forall u,v \in H_0^1(\Omega).
\end{equation}

In order to define higher derivatives of functions with multi variables,
we introduce the following multi-index
$\alpha = (\alpha_1,\cdots,\alpha_d)$ and the corresponding differential operator:
\begin{equation}\label{Multi_Diff}
D^{\alpha} = \partial^{\alpha_1}_{x_1}\partial^{\alpha_2}_{x_2}\cdots\partial^{\alpha_d}_{x_d}.
\end{equation}
Furthermore, we say $\alpha \le \beta$ if and only if $\alpha_i \le \beta_i$, $i = 1,\cdots, d$.
And when $\alpha \ge \beta$, we denote $\alpha - \beta = (\alpha_1 - \beta_1, \cdots, \alpha_d - \beta_d)$.
For derivative of the product of two functions, we have
\begin{equation}\label{r_diff_product}
D^{\alpha}(fg) = \sum_{i=0}^{|\alpha|}\sum_{|\beta|=i,\beta+\gamma=\alpha} D^{\beta}f D^{\gamma}g,
\end{equation}
where $|\alpha|=\alpha_1+\cdots+\alpha_d$.

\subsection{Some assumptions on the finite element spaces}
Following \cite{XuZhLP2G}, we present some assumptions on the finite element spaces and then
define the corresponding finite element approximation for the problem
(\ref{Weak_Form_Global}).

First we generate a shape-regular
decomposition $\mathcal{T}_h(\Omega)$ for the computing domain $\Omega\subset \mathcal{R}^d\
(d=2,3)$ into triangles or rectangles for $d=2$ (tetrahedrons or
hexahedrons for $d=3$). The diameter of a cell $K\in\mathcal{T}_h(\Omega)$
is denoted by $h_K$. The mesh size function is denoted by $h(x)$ whose value
is the diameter $h_K$ of the element $K$ including $x$.

Now, we state the following assumption for the mesh considered in this paper:\\
{\it {\bf A.0}.\
There exists $\gamma>1$ such that
\begin{eqnarray}\label{Mesh_Size_Condition}
h_{\Omega}^{\gamma}\le C h(x),\ \ \ \ \forall x\in\Omega,
\end{eqnarray}
where $h_{\Omega}=\max_{x\in\Omega}h(x)$ is the
largest mesh size of $\mathcal{T}_h(\Omega)$ and $C$ is a constant independent of $\Omega$ and $h(x)$.
}

Based on the triangulation $\mathcal{T}_h(\Omega)$, we define the finite element space
$S_h(\Omega)\subset H^1(\Omega)$ and $S_h^0(\Omega) = S_h(\Omega) \cap H_0^1(\Omega)$.
Given $G\subset \Omega$, we use $S_h(G)$ and $\mathcal{T}_h(G)$ to denote the restriction
of $S_h(\Omega)$ and $\mathcal{T}_h(\Omega)$ to $G$, respectively, and define
\begin{eqnarray}\label{FEM_G}
S_h^0(G)=\big\{v\in S_h(\Omega):\ {\rm supp}\ v\subset\subset G\big\}.
\end{eqnarray}
For any concerned subdomain $G\subset \Omega$ in this paper, we assume that it aligns with the
partition $\mathcal{T}_h(\Omega)$.

Now, we would like to state some assumptions on the finite element space.
The constants $C$ appeared here and after are independent of the scale of $\Omega$ and mesh parameters.\\
{\it {\bf A.1}.\ (Approximation).\
For any $w\in H_0^1(\Omega)$, we have
\begin{eqnarray}\label{Approximation}
\inf_{v\in S_{h}^0(\Omega)}\big(\|h^{-1}(w-v)\|_{0,\Omega}+\|w-v\|_{1,\Omega}\big)=o(1),
\end{eqnarray}
as $h_{\Omega}\rightarrow 0$.
}\\
{\it
{\bf A.1'}.\ (Approximation).\ There exists $r\geq 1$ such that for any $w\in H_0^1(\Omega)$,
\begin{eqnarray}\label{Approximation_II}
\inf_{v\in S_{h}^0(\Omega)}\big(h_{\Omega}^{-1}\|w-v\|_{0,\Omega}+\|w-v\|_{1,\Omega}\big)\le C h_{\Omega}^s\|w\|_{1+s,\Omega},
\ \ \ 0\leq s\leq r.
\end{eqnarray}
}\\
{\it
{\bf A.2}.\ (Inverse Estimate).\ For any $v\in S_h(\Omega_0)$,
\begin{eqnarray}\label{Inverse_Estimate}
\|v\|_{1,\Omega_0}\le C h_{\Omega}^{-1}\|v\|_{0,\Omega_0}.
\end{eqnarray}
}\\
{\it
{\bf A.3}.\ (Superapproximation).\ For $G\subset \Omega_0$, let $\omega\in C_0^{\infty}(\Omega)$ with ${\rm supp}\  \omega\subset\subset G$. Then for any $w\in S_h(G)$, there is $v\in S_{h}^0(G)$ such that
\begin{eqnarray}\label{Superapproximation}
\|\omega w-v\|_{1,G}\le C h_G\|w\|_{1,G}.
\end{eqnarray}}
To show the reasonability of the above assumptions,
we state the normal Lagrange finite element spaces which satisfies the above assumptions, i.e.,
\begin{equation}\label{Lagrange_r}
S_h(\Omega) = \big\{ v \in C(\bar{\Omega}): v|_K \in P_r(K),\ \forall K\in \mathcal{T}_h(\Omega) \big\},
\end{equation}
where $\mathcal{P}_r(K)$ denote the space of polynomials of degree not greater
than the positive integer $r$.

Now, we can to investigate the new versions of Assumptions {\bf A.1}, {\bf A.1'}, {\bf A.2} and {\bf A.3}
on the general subdomain scales.
For this aim, we need to introduce the affine mapping which transforms the general domain $\Omega_0$ to the
reference domain $\widehat{\Omega}_0$ with size $1$.
The affine mapping can be defined as follows:
\begin{equation}\label{transform}
F:\ \Omega_0 \rightarrow \widehat{\Omega}_0,\ \ \ \ x \rightarrow \xi := \frac{x-x_0}{d_{\Omega_0}},
\end{equation}
where $x_0$ is any inner point of $\Omega_0$. Through this map, $K$ and $\mathcal{T}_h(\Omega_0)$ are transformed to $\widehat{K}$ and $\widehat{\mathcal{T}}_h(\widehat{\Omega}_0)$, respectively.
It is obvious that
\begin{eqnarray*}
\frac{h_{\Omega_0}}{d_{\Omega_0}} = \frac{h_{\widehat{\Omega}_0}}{d_{\widehat{\Omega}_0}}.
\end{eqnarray*}
We define $\widehat{u}(\xi) = u(x)$ with $\xi=\frac{x-x_{0}}{d_{\Omega_0}}$ for $u(x)$ with $x\in \Omega_0$.
Then it is naturally that $\widehat{uv} = \hat{u}\hat{v}$. Similarly to (\ref{Multi_Diff}), we also define
\begin{equation}\label{Multi_Diff_Hat}
\widehat{D}^{\alpha} = \partial^{\alpha_1}_{\xi_1}\partial^{\alpha_2}_{\xi_2}\cdots\partial^{\alpha_d}_{\xi_d}.
\end{equation}
 It is easy to derive that $\widehat{D}^{\alpha}\widehat{u}(\xi) = d_{\Omega_0}^{|\alpha|}D^{\alpha}u(x)$ and
 $D^{\alpha}u(x) = d_{\Omega_0}^{-|\alpha|}\widehat{D}^{\alpha}\widehat{u}(\xi)$. Then we can define the corresponding $\widehat{S}_h(\widehat{\Omega}_0)$ which can be viewed as the transformation of $S_h(\Omega_0)$ through the map (\ref{transform}).

\begin{proposition}\label{New_A3}
If we take $S_h(\Omega)$ as in (\ref{Lagrange_r}), then assumptions A.1, A.1' and A.2 hold.
Assumption A.3 should be changed to the following version:

{\it
{\bf A.3}.\ (Superapproximation).\
\ For $G\subset \Omega_0$, let $\omega\in C_0^{\infty}(\Omega)$ with ${\rm supp}\  \omega\subset\subset G$. Then for any $w\in S_h(G)$, there is $v\in S_{h}^0(G)$ such that
\begin{eqnarray}\label{Superapproximation_New}
\|\omega w-v\|_{1,G}\le C d_{\Omega_0}^{-1}\Big(\frac{h_{\Omega_0}}{d_{\Omega_0}}\Big)^{r}\|w\|_{0,\Omega_0}
+C\frac{h_{\Omega_0}}{d_{\Omega_0}}\|w\|_{1,\Omega_0}.
\end{eqnarray}}
\end{proposition}
\begin{proof}
First, it is obvious that the space $S_h(\Omega)$ satisfies Assumptions {\bf A.1}, {\bf A.1'}, {\bf A.2} (c.f. \cite{Ciarlet}).
Here we mainly concern the proof of Assumption {\bf A.3}. From $\omega \in C_0^{\infty}(\Omega_0)$, $\widehat{\omega} \in C_0^{\infty}({\widehat{\Omega}}_0)$, Assumption {\bf A.1'} and polynomial interpolation theory in \cite{Ciarlet}, there exist
$\widehat{v} \in S_h^0(\widehat{\Omega}_0)$  such that
\begin{eqnarray}\label{Interpolation_Hat}
|\widehat{\omega}\widehat{w} - \widehat{v}|_{1,\widehat{\Omega}_0}^2
&\le& Ch_{\widehat{\Omega}_0}^{2r}\left(\sum_{\widehat{K}\in\widehat{T}_h(\widehat{\Omega}_0)}
|\widehat{\omega}\widehat{w}|_{1+r,\widehat{K}}^2\right),\\
\|\widehat{\omega}\widehat{w} - \widehat{v}\|_{0,\widehat{\Omega}_0}^2
 &\le& Ch_{\widehat{\Omega}_0}^{2+2r}\left(\sum_{\widehat{K}\in\widehat{\mathcal{T}}_h(\widehat{\Omega}_0)}
|\widehat{\omega}\widehat{w}|_{1+r,\widehat{K}}^2\right),
\end{eqnarray}
where $C$ is a constant independent of $\widehat{\Omega}_0$. Frome Leibnitz formula
(\ref{r_diff_product}) and $\widehat{w}|_{\widehat{K}} \in \widehat{P}^r(\widehat{K})$ on any element $\widehat{K} \in \widehat{\mathcal{T}}_h(\widehat{\Omega}_0)$, the following inequalities hold
\begin{eqnarray}\label{r_Diff_Hat}
&&\sum_{\widehat{K}\in \widehat{\mathcal{T}}_h(\widehat{\Omega}_0)}
|\widehat{\omega}\widehat{w}|_{1+r,\widehat{K}}^2 = \sum_{\widehat{K}\in
\widehat{\mathcal{T}}_h(\widehat{\Omega}_0)}\int_{\widehat{K}} \sum_{|\alpha|=1+r}
|\widehat{D}^{\alpha}(\widehat{\omega}\widehat{w})|^2 \ d\widehat{K} \nonumber \\
&=& \sum_{\widehat{K}\in \widehat{\mathcal{T}}_h(\widehat{\Omega}_0)}
 \int_{\widehat{K}} \sum_{|\alpha|=1+r}
  \Big|\sum_{i=0}^{1+r}\sum_{|\beta|=i,\beta \le
   \alpha, \gamma = \alpha-\beta} \widehat{D}^{\beta}\widehat{w} \widehat{D}^{\gamma}\widehat{\omega}\Big|^2 d\widehat{K}\nonumber\\
&=& \sum_{\widehat{K}\in \widehat{\mathcal{T}}_h(\widehat{\Omega}_0)} \int_{\widehat{K}} \sum_{|\alpha|=1+r} \Big|\sum_{i=0}^{r}\sum_{|\beta|=i,\beta \le \alpha, \gamma = \alpha-\beta} \widehat{D}^{\beta}\widehat{w} \widehat{D}^{\gamma}\widehat{\omega}\Big|^2 d\widehat{K} \nonumber\\
&\le& \sum_{\widehat{K}\in \widehat{\mathcal{T}}_h(\widehat{\Omega}_0)} \int_{\widehat{K}} \sum_{|\alpha|=1+r} C_{r,d} \sum_{i=0}^{r}
\sum_{|\beta|=i,\beta \le \alpha, \gamma = \alpha-\beta}
|\widehat{D}^{\beta}\widehat{w}|^2 |\widehat{D}^{\gamma}\widehat{\omega}|^2 d\widehat{K}\nonumber\\
&=&  \sum_{\widehat{K}\in \widehat{\mathcal{T}}_h(\widehat{\Omega}_0)} \int_{\widehat{K}} C_{r,d} \sum_{i=0}^{r}\sum_{|\beta|=i} |\widehat{D}^{\beta}\widehat{w}|^2 \sum_{|\gamma + \beta|=1+r} |\widehat{D}^{\gamma}\widehat{\omega}|^2 \ d\widehat{K}\nonumber \\
&\le& C_{r,d} \sum_{\widehat{K}\in \widehat{\mathcal{T}}_h(\widehat{\Omega}_0)} \int_{\widehat{K}} \sum_{i=0}^{r}\sum_{|\beta|=i} |\widehat{D}^{\beta}\widehat{w}|^2  d\widehat{K}.
\end{eqnarray}
We take $v(x) = \widehat{v}(\xi)$ with $\xi = \frac{x-x_0}{d_{\Omega_0}}$ and claim that $v$ is
the desired function in Assumption {\bf A.3}.
In fact, by changing variables and combing (\ref{Interpolation_Hat}) and (\ref{r_Diff_Hat}), we have
\begin{eqnarray}\label{break_11_seminorm}
&&|\omega w - v|_{1,\Omega_0}^2 = \int_{\Omega_0}|\nabla(\omega w - v)|^2 \
d\Omega_0\nonumber \\
&=& d_{\Omega_0}^{-2} \frac{|\Omega_0|}{|\widehat{\Omega}_0|} \int_{\widehat{\Omega}_0}|\widehat{\nabla}(\widehat{\omega w} - \widehat{v})|^2 \ d\widehat{\Omega}_0
\le C d_{\Omega_0}^{d-2} h_{\widehat{\Omega}_0}^{2r} \sum_{\widehat{K}\in \widehat{\mathcal{T}}_h(\widehat{\Omega}_0)}
|\widehat{\omega}\widehat{w}|_{1+r,\widehat{K}}^2\nonumber \\
&\le& C d_{\Omega_0}^{d-2}\Big(\frac{h_{\Omega_0}}{d_{\Omega_0}}d_{\widehat{\Omega}_0}\Big)^{2r}  C_{r,d} \sum_{\widehat{K}\in \widehat{\mathcal{T}}_h(\widehat{\Omega}_0)} \int_{\widehat{K}} \sum_{i=0}^{r}\sum_{|\beta|=i} |\widehat{D}^{\beta}\widehat{w}|^2 d\widehat{K} \nonumber \\
&\le& C d_{\Omega_0}^{d-2}\Big(\frac{h_{\Omega_0}}{d_{\Omega_0}}d_{\widehat{\Omega}_0}\Big)^{2r}  C_{r,d} \sum_{\widehat{K}\in \widehat{\mathcal{T}}_h(\widehat{\Omega}_0)} \frac{|\widehat{K}|}{|K|} \int_{K} \sum_{i=0}^{r}\sum_{|\beta|=i}d_{\Omega_0}^{2i} |D^{\beta}w|^2 d\widehat{K} \nonumber \\
&\le& C d_{\Omega_0}^{d-2-2r}h_{\Omega_0}^{2r} C_{r,d}\max_{K \in \mathcal{T}_h(\Omega_0)}\frac{|\widehat{K}|}{|K|} \sum_{K \in \mathcal{T}_h(\Omega_0)} \int_{K} \sum_{i=0}^{r}\sum_{|\beta|=i}d_{\Omega_0}^{2i} |D^{\beta}w|^2
d\widehat{K}\nonumber\\
&\le& C d_{\Omega_0}^{-2-2r}h_{\Omega_0}^{2r} C_{r,d}
\sum_{K \in \mathcal{T}_h(\Omega_0)}
\sum_{i=0}^{r}d_{\Omega_0}^{2i}\left(\sum_{|\beta|=i} \|D^{\beta}w\|_{0,K}^2\right).
\end{eqnarray}
Together with the inverse inequality, (\ref{break_11_seminorm}) can be reduced to
\begin{eqnarray}\label{super_1_seminorm}
|\omega w - v|_{1,\Omega_0}^2 &\le& C d_{\Omega_0}^{-2-2r}h_{\Omega_0}^{2r} C_{r,d}
\sum_{K \in \mathcal{T}_h(\Omega_0)}
\Big(\|w\|_{0,K}^2 + d_{\Omega_0}^{2} |w|_{1,K}^2 + d_{\Omega_0}^{4}h_{\Omega_0}^{-2}
|w|_{1,K}^2 \nonumber\\
&&\ \ \ \ \quad\quad\quad\quad\quad\quad\quad
+ \cdots + d_{\Omega_0}^{2r}h_{\Omega_0}^{-2(r-1)} |w|_{1,K}^2 \Big) \nonumber\\
&\le& C d_{\Omega_0}^{-2-2r}h_{\Omega_0}^{2r} C_{r,d}
\sum_{K \in \mathcal{T}_h(\Omega_0)}|w|_{1,K}^2
\Big(d_{\Omega_0}^{2} + d_{\Omega_0}^{4}h_{\Omega_0}^{-2} + \cdots \nonumber\\
&&\quad\quad\quad\quad\quad  + d_{\Omega_0}^{2r}h_{\Omega_0}^{-2(r-1)} \Big)
+C C_{r,d} d_{\Omega_0}^{-2}\Big(\frac{h_{\Omega_0}}{d_{\Omega_0}}\Big)^{2r}\|w\|_{0,\Omega_0}^2 \nonumber\\
&\le& C C_{r,d} \Big(\frac{h_{\Omega_0}}{d_{\Omega_0}}\Big)^{2}
\|w\|_{1,\Omega_0}^2
\Big(\frac{h_{\Omega_0}^{2r-2}}{d_{\Omega_0}^{2r-2}} + \frac{h_{\Omega_0}^{2r-4}}{d_{\Omega_0}^{2r-4}} + \cdots + 1 \Big)\nonumber\\
&&\quad\quad\quad\quad +C C_{r,d} d_{\Omega_0}^{-2}\Big(\frac{h_{\Omega_0}}{d_{\Omega_0}}\Big)^{2r}\|w\|_{0,\Omega_0}^2\nonumber\\
&\le&CC_{r,d} d_{\Omega_0}^{-2}\Big(\frac{h_{\Omega_0}}{d_{\Omega_0}}\Big)^{2r}\|w\|_{0,\Omega_0}^2
+C C_{r,d} \Big(\frac{h_{\Omega_0}}{d_{\Omega_0}}\Big)^{2}\|w\|_{1,\Omega_0}^2.
\end{eqnarray}
From Poincar\'{e} inequality, we have
\begin{eqnarray}\label{super_0_norm}
\|\omega w - v\|_{0,\Omega_0} &\leq& Cd_{\Omega_0}|\omega w - v|_{1,\Omega_0}.
\end{eqnarray}
Then combining (\ref{super_1_seminorm}) and (\ref{super_0_norm}),
we obtain the desired result (\ref{Superapproximation_New}) and the proof is complete.
\end{proof}

\section{Local a priori estimate}

In this section, we derive a new local a priori estimate which is
dependent on the subdomain scale and is different from the one in \cite{XuZhLP2G}
where the local a priori
estimate are provided for the case with the subdomain scale being $\mathcal{O}(1)$.
The local estimate here is for the general subdomain scales.

 The following lemma is the same as \cite[Lemma 3.1]{XuZhLP2G}. But here we need to prove it for
 the general scale subdomains $\Omega_0$.
\begin{lemma}\label{Technique_Lem}
Let $D \subset\subset \Omega_0$, and let $\omega \in C_0^{\infty}(\Omega)$ be such that
$supp\ \omega \subset\subset \Omega_0$. Then
\begin{equation}\label{Technique}
a_0(\omega w, \omega w) \le 2a(w, \omega^2 w) + C\|w\|_{0,\Omega_0}^2, \ \ \forall w \in H_0^1(\Omega).
\end{equation}
\end{lemma}
\begin{proof}
With integration by parts, we have the following identity
\begin{eqnarray}\label{Identity}
&a_0(\omega w,\omega w)&= a(w,\omega^2 w) - N(\omega w,\omega w) +
\int_{\Omega}\sum_{j=1}^d b_j \frac{\partial \omega}{\partial x_j}\omega w^2 d\Omega\nonumber \\
&&\hskip-1cm + \int_{\Omega}\sum_{i,j=1}^d a_{ij}\left(\Big(\frac{\partial\omega}{\partial x_i}\frac{\partial(\omega w)}{\partial x_j} - \frac{\partial\omega}{\partial x_j}\frac{\partial(\omega w)}{\partial x_i}\Big)w + \frac{\partial\omega}{\partial x_i}\frac{\partial\omega}{\partial x_j}w^2\right)d\Omega.
\end{eqnarray}
Let us define
\begin{eqnarray*}
T_1(\omega,w) &=& \int_{\Omega}\sum_{j=1}^d b_j \frac{\partial \omega}{\partial x_j}\omega w^2d\Omega,
\end{eqnarray*}
and
\begin{eqnarray*}
T_2(\omega,w) &=& \int_{\Omega}\sum_{i,j=1}^d a_{ij}\left(\Big(\frac{\partial\omega}{\partial x_i}\frac{\partial(\omega w)}{\partial x_j} - \frac{\partial\omega}{\partial x_j}\frac{\partial(\omega w)}{\partial x_i}\Big)w + \frac{\partial\omega}{\partial x_i}\frac{\partial\omega}{\partial x_j}w^2\right)d\Omega.
\end{eqnarray*}
Then we can rewrite the identity (\ref{Identity}) as follows
\begin{equation*}
a_0(\omega w,\omega w)  = a(w,\omega^2 w) - N(\omega w,\omega w) + T_1(\omega,w) + T_2(\omega,w).
\end{equation*}
With the transform operator $F$ defined in (\ref{transform}) and following the bilinear form (\ref{Bilinear_Form}), we define
\begin{equation}\label{Bilinear_Form_Hat}
\widehat a_0(u,v) = \int_{\widehat\Omega}\sum_{i,j=1}^d \widehat a_{ij}\frac{\widehat\partial \widehat u}{\partial\xi_i}\frac{\widehat\partial \widehat v}{\partial \xi_j}d\widehat{\Omega}\ \  {\rm{and}} \ \
\widehat N(\widehat u,\widehat v) = \int_{\widehat\Omega}\left(\sum_{i=1}^d \widehat b_i \frac{\widehat\partial \widehat u}{\partial \xi_i}v + \widehat{\phi} \widehat u \widehat v\right)d\widehat{\Omega}.
\end{equation}
Thus
\begin{equation*}
\widehat a(\widehat u,\widehat v) = \widehat a_0(\widehat u,\widehat v) + \widehat N(\widehat u,\widehat v).
\end{equation*}
Similarly, we define
\begin{eqnarray*}
\widehat T_1(\widehat \omega,\widehat w) &=& \int_{\widehat \Omega}\sum_{j=1}^d \widehat b_j \frac{\widehat \partial \widehat \omega}{\partial \xi_j}\widehat \omega \widehat w^2d\widehat{\Omega},\\
\widehat T_2(\widehat \omega,\widehat w) &=& \int_{\widehat \Omega}\left(\sum_{i,j=1}^d \widehat a_{ij}\Big(\frac{\widehat \partial\omega}{\partial \xi_i}\frac{\widehat \partial(\widehat \omega \widehat w)}{\partial \xi_j} - \frac{\widehat \partial\widehat \omega}{\partial \xi_j}\frac{\widehat \partial(\widehat \omega \widehat w)}{\partial \xi_i}\Big)\widehat w + \frac{\widehat \partial\widehat \omega}{\partial \xi_i}\frac{\widehat \partial\widehat \omega}{\partial \xi_j}\widehat w^2\right)d\widehat{\Omega}.
\end{eqnarray*}
Then the following identity holds:
\begin{equation*}
\widehat a_0(\widehat \omega \widehat w,\widehat \omega \widehat w)  = \widehat a(\widehat w,\widehat \omega^2 \widehat w) - \widehat N(\widehat \omega \widehat w,\widehat \omega \widehat w) + \widehat T_1(\widehat \omega,\widehat w) + \widehat T_2(\widehat \omega,\widehat w).
\end{equation*}
By changing variable, we have
\begin{eqnarray*}
a_0(\omega w,\omega w)  &=& d_{\Omega_0}^{-2} \frac{|\Omega_0|}{|\widehat{\Omega}_0|}\widehat a_0(\widehat \omega \widehat w,\widehat \omega \widehat w)  \\
&\le& C d_{\Omega_0}^{d-2}\big(\widehat a(\widehat w,\widehat \omega^2 \widehat w) - \widehat N(\widehat \omega \widehat w,\widehat \omega \widehat w) + \widehat T_1(\widehat \omega,\widehat w) + \widehat T_2(\widehat \omega,\widehat w)\big) \\
&\le& C\big(a(w,\omega^2 w) - N(\omega w,\omega w) + T_1(\omega,w) + T_2(\omega,w)\big),
\end{eqnarray*}
where $C$ is a constant independent of the scale of $\Omega_0$.
Then the rest of the proof is the same as the proof of \cite[Lemma 3.1]{XuZhLP2G}.
\end{proof}
Now, we come give the local a priori estimate for the general scale domains. 
\begin{theorem}\label{Local_Estimate_Thm}
Suppose that $f \in H^{-1}(\Omega)$ and $D \subset\subset \Omega_0$.
If Assumptions {\bf A.0, A.1, A.2} and the new {\bf A.3} in Proposition \ref{New_A3} hold 
and $w \in S_h(\Omega_0)$ satisfies
\begin{eqnarray}\label{Weak_Form_Local}
a(w,v) = f(v), \ \ \ \forall v \in H_0^1(\Omega),
\end{eqnarray}
then the following local estimate holds
\begin{eqnarray}\label{Local_Estimate}
\|w\|_{1,D} \le C\left(\varepsilon^{\frac{p+1}{2}}h_{\Omega_0}^{-1}\|w\|_{0,\Omega_0}
+\sum_{j=0}^{p}\varepsilon^j\big(\|f\|_{-1,\Omega_0}+\|w\|_{0,\Omega_0}\big)\right),
\end{eqnarray}
where $C$ is a constant independent of $D$ and the mesh size, $p$ is the number of mesh layer from $D$ to
$\Omega_0$, $\varepsilon$ is defined by
\begin{eqnarray}
\varepsilon=\left(d_{\Omega_0}^{-2}\Big(\frac{h_{\Omega_0}}{d_{\Omega_0}}\Big)^{2r}
+\frac{h_{\Omega_0}}{d_{\Omega_0}}\right)^{\frac{1}{2}},
\end{eqnarray}
and  $\|f\|_{-1,\Omega_0}$ is defined as follows
\begin{equation*}
\|f\|_{-1,\Omega_0} = \sup_{\varphi \in H_0^1(\Omega_0),\|\varphi\|_{1,\Omega_0}=1} f(\varphi).
\end{equation*}
\end{theorem}
\begin{proof}
Let $p$ be an integer such that there exist $\Omega_j$ $(j=1,2,\cdots,p)$ satisfying
\begin{equation*}
D \subset\subset \Omega_p \subset\subset \Omega_{p-1} \subset\subset \cdots \subset\subset \Omega_1 \subset\subset \Omega_0.
\end{equation*}
Choose $D_1 \subset \Omega$ satisfying $D \subset\subset D_1 \subset\subset \Omega_p$ and $\omega \in C_0^{\infty}(\Omega)$ such that $\omega \equiv 1$ on $\bar{D}_1$ and $supp\ \omega \subset\subset \Omega_p$. Then From (\ref{positive_definite_A}), (\ref{Bounded_Operator}), (\ref{Weak_Form_Local}),
(\ref{Technique}) and Assumption {\bf A.3}, we have
\begin{eqnarray}\label{to_kick_back}
& & \|w\|_{1,D}^2 = \|\omega w\|_{1,D}^2 \le \|\omega w\|_{1,\Omega_p}^2 \nonumber\\
&\le & Ca_0(\omega w, \omega w) \le C\big(a(w, \omega^2 w) + \|w\|_{0,\Omega_p}^2\big) \nonumber\\
&\le & C\big(a(w, \omega^2 w - v) + f(v) + \|w\|_{0,\Omega_p}^2\big) \nonumber\\
&\le & C\big(\|w\|_{1,\Omega_p} \|\omega^2 w - v\|_{1,\Omega_p}
 + \|f\|_{-1,\Omega_0}\|v\|_{1,\Omega_p} + \|w\|_{0,\Omega_p}^2 \big) \nonumber\\
&\leq& C\Big(\|w\|_{1,\Omega_p}\|\omega^2w-v\|_{1,\Omega_p}+\|\omega^2w-v\|_{1,\Omega_p}
+\|f\|_{-1,\Omega_0}+\|w\|_{0,\Omega_0}\Big)\nonumber\\
&&\ \ \ \ \ \ +\frac{1}{2}\|\omega w\|_{1,\Omega_p}^2\nonumber\\
&\leq& C\left(d_{\Omega_0}^{-1}\Big(\frac{h_{\Omega_0}}{d_{\Omega_0}}\Big)^r\|w\|_{0,\Omega_p}\|w\|_{1,\Omega_p}+
\frac{h_{\Omega_0}}{d_{\Omega_0}}\|w\|_{1,\Omega_p}^2+\|f\|_{-1,\Omega_0}^2+\|w\|_{0,\Omega_0}^2\right)\nonumber\\
&&\ \ \ \ \ \ +\frac{1}{2}\|\omega w\|_{1,\Omega_p}^2\nonumber\\
&\leq& C\left(\Big(d_{\Omega_0}^{-2}\Big(\frac{h_{\Omega_0}}{d_{\Omega_0}}\Big)^{2r}
+\frac{h_{\Omega_0}}{d_{\Omega_0}}\Big)\|w\|_{1,\Omega_p}^2+\|f\|_{-1,\Omega_0}^2
+\|w\|_{0,\Omega_0}^2\right)\nonumber\\
&&\ \ \ \ \ \ +\frac{1}{2}\|\omega w\|_{1,\Omega_p}^2.
\end{eqnarray}
With an application of kick-back argument to (\ref{to_kick_back}), we have
\begin{equation*}
\|w\|_{1,D}^2 \le C\left(\Big(d_{\Omega_0}^{-2}\Big(\frac{h_{\Omega_0}}{d_{\Omega_0}}\Big)^{2r}
+\frac{h_{\Omega_0}}{d_{\Omega_0}}\Big)\|w\|_{1,\Omega_p}+\|f\|_{-1,\Omega_0}^2+\|w\|_{0,\Omega_0}^2\right).
\end{equation*}
Thus, the following estimate holds
\begin{equation*}
\|w\|_{1,D} \le C\left(\Big(d_{\Omega_0}^{-2}\Big(\frac{h_{\Omega_0}}{d_{\Omega_0}}\Big)^{2r}
+\frac{h_{\Omega_0}}{d_{\Omega_0}}\Big)^{\frac{1}{2}}\|w\|_{1,\Omega_p}
+\|f\|_{-1,\Omega_0}+\|w\|_{0,\Omega_0}\right).
\end{equation*}
Note that $D$ can be viewed as $\Omega_{p+1}$, the above process can be taken recursively and leads to
\begin{equation*}
\|w\|_{1,\Omega_j} \le C\left(\Big(d_{\Omega_0}^{-2}\Big(\frac{h_{\Omega_0}}{d_{\Omega_0}}\Big)^{2r}
+\frac{h_{\Omega_0}}{d_{\Omega_0}}\Big)^{\frac{1}{2}}\|w\|_{1,\Omega_{j-1}}
+\|f\|_{-1,\Omega_0}+\|w\|_{0,\Omega_0}\right).
\end{equation*}
Combining the above inequalities and the inverse inequality, we have
\begin{eqnarray}
\|w\|_{1,D} &\le& C\left(\varepsilon^{\frac{p+1}{2}}\|w\|_{1,\Omega_0}+\sum_{j=0}^{p}\varepsilon^j\big(
\|f\|_{-1,\Omega_0}+\|w\|_{0,\Omega_0}\big)\right)\nonumber\\
&\leq& C\left(\varepsilon^{\frac{p+1}{2}}h_{\Omega_0}^{-1}\|w\|_{0,\Omega_0}
+\sum_{j=0}^{p}\varepsilon^j\big(\|f\|_{-1,\Omega_0}+\|w\|_{0,\Omega_0}\big)\right).
\end{eqnarray}
This is the  desired result and the proof is complete.
\end{proof}
From the delicate local a priori estimate (\ref{Local_Estimate}), we can find that the 
usual local estimate
\begin{eqnarray*}
\|w\|_{1,D}&\leq C\big(\|w\|_{0,\Omega_0}+\|f\|_{-1,\Omega_0}\big)
\end{eqnarray*}
holds only on the case where the scale of $\Omega_0$ is $\mathcal{O}(1)$.
It means that the number of subdomains should be $\mathcal{O}(1)$
and the speed up rate of the  parallel technique based on local a priori estimate
also should be $\mathcal{O}(1)$.

\section{Concluding remarks}
In this note, we investigate the dependence of the local a priori estimates on the scale of the subdomains.
As we know, some domain decomposition and parallel techniques depend on the local a priori 
estimate of the finite element method. From the derived local estimate
(\ref{Local_Estimate}), we can find that the local estimate depends on the scale
of the subdomains. This dependence push a constraint to the speed up rate of the parallel
technique based on the local error estimate. This is why we consider this problem here
and the derived results may give some hints for constructing the
domain decomposition techniques to solve partial differential equations.

\end{document}